\documentclass[11pt]{article}

\usepackage[a4paper,margin=1in]{geometry}
\usepackage{amsmath,amssymb,amsthm,mathtools}
\usepackage{booktabs}
\usepackage{enumitem}

\newtheorem{theorem}{Theorem}[section]
\newtheorem{lemma}[theorem]{Lemma}
\newtheorem{proposition}[theorem]{Proposition}
\newtheorem{corollary}[theorem]{Corollary}
\newtheorem{remark}[theorem]{Remark}

\DeclareMathOperator{\Ree}{Re}

\title{The Sharp Sadov Constant and Local Spectral Stability for Shapiro--Diananda Cyclic Sums}
\author{Denis Sheremet\\
Independent researcher\\
\texttt{denis.a.sheremet@gmail.com}}
\date{June 2026}

\begin{document}
\maketitle

\begin{abstract}
We study cyclic Shapiro--Diananda sums
\[
        S_{n,k}(x)=\sum_{i=1}^n \frac{x_i}{x_{i+1}+\cdots+x_{i+k}},
        \qquad x_i>0,
\]
with cyclic indices.  First, we determine their exact local quadratic behaviour at the equal point.  Writing \(x_i=e^{u_i}\) and removing the common scaling direction, the Hessian becomes a circulant quadratic form.  Its Fourier symbol is
\[
        \gamma_{m,k}
        =
        \frac{|\lambda_{m,k}|^2-k\Ree\lambda_{m,k}}{k^3},
        \qquad
        \lambda_{m,k}=\sum_{j=1}^k e^{2\pi i mj/n}.
\]
Thus the equal point is a strict local minimum, quadratically degenerate, or a saddle according as \(\min_m\gamma_{m,k}\) is positive, zero, or negative.  We classify the cases \(k=2\) and \(k=3\), prove periodic equality families, and show that for every \(k\ge3\) the equal point is a saddle whenever \(n>12k\).  Second, combining Sadov's lower bound with an explicit asymptotic construction, we prove that the sharp global Sadov constant for the normalized Diananda sums equals \(\log2\).  This determines the global infimum over all admissible triples \((n,k,x)\); the fixed-\(k\) asymptotic constants remain separate problems.
\end{abstract}

Fourier modes; circulant Hessian.\par

\section{Introduction}

The cyclic sums
\[
        \sum_{i=1}^n \frac{x_i}{x_{i+1}+x_{i+2}}
\]
are classically associated with Shapiro's cyclic inequality.  More generally, the sums
\[
        S_{n,k}(x)
        =\sum_{i=1}^n
        \frac{x_i}{x_{i+1}+\cdots+x_{i+k}}
\]
are often called Shapiro--Diananda type cyclic sums.  Their global lower bounds have a long history and are subtle: the equal point is not always globally extremal, and counterexamples are known in several ranges.  Sadov studied global lower bounds for such sums and related generalized Shapiro--Diananda and graphic sums; see, for example, \cite{Sadov2016,Sadov2021,Sadov2022}.

The present note has two parts.  The first part studies the local question:

\begin{quote}
What is the exact local quadratic behaviour of \(S_{n,k}\) near the equal point \(x_1=\cdots=x_n\)?
\end{quote}

The second part uses a different asymptotic construction to identify the sharp global Sadov constant.  This concerns the infimum over all admissible pairs \(n\ge k\) and all positive cyclic vectors; it does not determine the separate fixed-\(k\) asymptotic constants.  Thus the two parts describe two opposite regimes of the same cyclic sums: the local geometry near the equal configuration and the global concentrating regime that attains the sharp constant.

Since \(S_{n,k}\) is homogeneous, the natural local coordinates are logarithmic:
\[
        x_i=e^{u_i},
        \qquad
        \sum_{i=1}^n u_i=0.
\]
In these coordinates the equal point is \(u=0\).  We show that
\[
        S_{n,k}(e^u)
        =\frac nk+Q_{n,k}(u)+O(\|u\|^3),
\]
where \(Q_{n,k}\) is a circulant quadratic form.  This form is diagonalized by the discrete Fourier basis, and hence its sharp local quadratic coefficient is given explicitly by a finite Fourier minimum.

This gives a complete local trichotomy:
\begin{itemize}[nosep]
\item if \(C_{n,k}^{\rm loc}>0\), the equal point is a strict local minimum and a sharp quadratic stability inequality holds;
\item if \(C_{n,k}^{\rm loc}=0\), there is no positive quadratic stability constant;
\item if \(C_{n,k}^{\rm loc}<0\), the equal point is a saddle, and \(S_{n,k}(x)<n/k\) for positive \(x\) arbitrarily close to equality.
\end{itemize}

The cases \(k=2\) and \(k=3\) are worked out completely.  For \(k=2\), the answer is governed by the parity of \(n\).  For \(k=3\), the answer is arithmetic in \(n\): only finitely many \(n\) give a strict local minimum, finitely many give quadratic degeneracy, and all remaining \(n\) give a saddle.  We also exhibit exact periodic equality families: if a common period divides both \(n\) and \(k\), then \(S_{n,k}=n/k\) identically on that periodic family.  In particular, \(\gcd(n,k)=1\) is necessary for strict local stability.  Finally, for every fixed \(k\ge3\), we prove the uniform sufficient condition \(n>12k\) for the equal point to be a saddle, and we also give a sharper computable finite-grid criterion for this transition.

\section{The cyclic sums and logarithmic coordinates}

Let \(n\ge2\) and \(1\le k\le n-1\).  This restriction is used throughout the local spectral part of the paper.  In the global Sadov constant section we return to Sadov's convention \(n\ge k\ge1\), where the endpoint case \(k=n\) is also allowed.  Throughout, indices are understood modulo \(n\).  Define
\[
        S_{n,k}(x)
        =
        \sum_{i=1}^n
        \frac{x_i}{x_{i+1}+\cdots+x_{i+k}},
        \qquad x_i>0.
\]
At the equal point \(x_1=\cdots=x_n\),
\[
        S_{n,k}(x)=\frac nk.
\]
The function \(S_{n,k}\) is homogeneous of degree zero, so common scaling of all \(x_i\) does not matter.  We write
\[
        x_i=e^{u_i}
\]
and work on the hyperplane
\[
        \sum_{i=1}^n u_i=0.
\]
Equivalently, for arbitrary positive \(x\), the relevant squared distance to equality is
\[
        \sum_{i=1}^n
        \left(\log x_i-\overline{\log x}\right)^2,
        \qquad
        \overline{\log x}=\frac1n\sum_{i=1}^n \log x_i.
\]

\section{Quadratic expansion}

For \(u=(u_1,\ldots,u_n)\), put
\[
        L_i=u_{i+1}+\cdots+u_{i+k}.
\]

\begin{lemma}[Quadratic expansion]
As \(u\to0\) with \(\sum_i u_i=0\),
\[
        S_{n,k}(e^u)
        =
        \frac nk+Q_{n,k}(u)+O(\|u\|^3),
\]
where
\[
        Q_{n,k}(u)
        =
        -\frac1{k^2}\sum_{i=1}^n u_iL_i
        +
        \frac1{k^3}\sum_{i=1}^n L_i^2.
\]
Equivalently,
\[
        Q_{n,k}(u)
        =
        -\frac1{k^2}\sum_{i=1}^n
        u_i(u_{i+1}+\cdots+u_{i+k})
        +
        \frac1{k^3}\sum_{i=1}^n
        (u_{i+1}+\cdots+u_{i+k})^2.
\]
\end{lemma}

\begin{proof}
For each \(i\),
\[
        \frac{e^{u_i}}{e^{u_{i+1}}+\cdots+e^{u_{i+k}}}
        =
        \frac{1+u_i+u_i^2/2+O(\|u\|^3)}
        {k+L_i+\frac12\sum_{j=1}^k u_{i+j}^2+O(\|u\|^3)}.
\]
Expanding the reciprocal denominator gives
\[
        \frac1{k+L_i+\frac12\sum_{j=1}^k u_{i+j}^2}
        =
        \frac1k-\frac{L_i}{k^2}
        -\frac1{2k^2}\sum_{j=1}^k u_{i+j}^2
        +\frac{L_i^2}{k^3}
        +O(\|u\|^3).
\]
Multiplying by the numerator, we obtain
\begin{align*}
        \frac{e^{u_i}}{e^{u_{i+1}}+\cdots+e^{u_{i+k}}}
        ={}&
        \frac1k+\frac{u_i}{k}-\frac{L_i}{k^2}
        +\frac{u_i^2}{2k}
        -\frac{u_iL_i}{k^2} \\
        &-\frac1{2k^2}\sum_{j=1}^k u_{i+j}^2
        +\frac{L_i^2}{k^3}
        +O(\|u\|^3).
\end{align*}
After summing over \(i\), the linear part vanishes because
\[
        \sum_i L_i=k\sum_i u_i=0.
\]
The pure square terms also cancel since
\[
        \sum_{i=1}^n\sum_{j=1}^k u_{i+j}^2
        =k\sum_{i=1}^n u_i^2.
\]
The remaining second-order terms are precisely \(Q_{n,k}(u)\).
\end{proof}

\section{Fourier diagonalization and the local stability constant}

Let
\[
        \theta_m=\frac{2\pi m}{n},
        \qquad m=0,1,\ldots,n-1,
\]
and define
\[
        \lambda_{m,k}=\sum_{j=1}^k e^{ij\theta_m}.
\]

\begin{lemma}[Fourier symbol]
On the non-constant Fourier mode \(m=1,\ldots,n-1\), the quadratic form \(Q_{n,k}\) has coefficient
\[
        \gamma_{m,k}
        =
        \frac{|\lambda_{m,k}|^2-k\Ree\lambda_{m,k}}{k^3}.
\]
Consequently, on the zero-mean subspace,
\[
        Q_{n,k}(u)
        \ge
        C_{n,k}^{\rm loc}\sum_{i=1}^n u_i^2,
\]
where
\[
        C_{n,k}^{\rm loc}
        =
        \min_{1\le m\le n-1}\gamma_{m,k}.
\]
The constant is sharp.
\end{lemma}

\begin{proof}
For the complex mode \(u_i=e^{i\theta_m i}\),
\[
        L_i=\sum_{j=1}^k u_{i+j}=\lambda_{m,k}u_i.
\]
The second term in \(Q_{n,k}\) contributes \(|\lambda_{m,k}|^2/k^3\) times the squared norm.  The first term contributes \(-\Ree\lambda_{m,k}/k^2\), because \(Q_{n,k}\) is real on real vectors and the real quadratic form sees the real part of the convolution eigenvalue.  Thus the coefficient is
\[
        -\frac{\Ree\lambda_{m,k}}{k^2}
        +\frac{|\lambda_{m,k}|^2}{k^3}
        =
        \frac{|\lambda_{m,k}|^2-k\Ree\lambda_{m,k}}{k^3}.
\]
The discrete Fourier basis diagonalizes all circulant quadratic forms, and the constant mode is excluded by the zero-mean condition.  The lower bound and sharpness follow by taking the minimum over non-constant modes.
\end{proof}

\begin{theorem}[Local spectral trichotomy]
Let \(1\le k\le n-1\).  Define
\[
        C_{n,k}^{\rm loc}
        =
        \min_{1\le m\le n-1}
        \frac{|\lambda_{m,k}|^2-k\Ree\lambda_{m,k}}{k^3}.
\]
Then the following hold.

\begin{enumerate}[label=(\roman*),nosep]
\item If \(C_{n,k}^{\rm loc}>0\), then the equal point is a strict local minimum modulo scaling.  More precisely, for every \(\varepsilon>0\) there is a neighbourhood of the equal point such that
\[
        S_{n,k}(x)-\frac nk
        \ge
        \left(C_{n,k}^{\rm loc}-\varepsilon\right)
        \sum_{i=1}^n
        \left(\log x_i-\overline{\log x}\right)^2.
\]
\item If \(C_{n,k}^{\rm loc}<0\), then the equal point is a saddle.  In particular, there are positive vectors \(x\) arbitrarily close to equality such that
\[
        S_{n,k}(x)<\frac nk.
\]
\item If \(C_{n,k}^{\rm loc}=0\), then there is no positive quadratic stability constant at the equal point.
\end{enumerate}
\end{theorem}

\begin{proof}
The first assertion follows from the expansion
\[
        S_{n,k}(e^u)-\frac nk
        =Q_{n,k}(u)+O(\|u\|^3)
\]
and the Fourier lower bound on \(Q_{n,k}\).  The error term is bounded by \(\varepsilon\|u\|^2\) in a sufficiently small neighbourhood.

If \(C_{n,k}^{\rm loc}<0\), choose a real Fourier mode \(u\) with \(Q_{n,k}(u)<0\).  Then
\[
        S_{n,k}(e^{tu})-\frac nk
        =t^2Q_{n,k}(u)+O(t^3)<0
\]
for all sufficiently small nonzero \(t\).

If \(C_{n,k}^{\rm loc}=0\), a positive quadratic lower bound cannot hold, because a Fourier mode attaining the zero coefficient would violate it to second order.  This is what we mean below by quadratic degeneracy; no assertion about the first nonzero higher-order term is made unless an explicit equality family is exhibited.
\end{proof}

It is often useful to write the symbol as a function of \(x=m/n\).  Since
\[
        \lambda_{m,k}
        =
        e^{\pi i(k+1)x}\frac{\sin(\pi kx)}{\sin(\pi x)},
        \qquad x=\frac mn,
\]
put
\[
        A_k(x)=\frac{\sin(\pi kx)}{\sin(\pi x)}.
\]
Then
\[
        \gamma_{m,k}=\Gamma_k\left(\frac mn\right),
\]
where
\[
        \Gamma_k(x)
        =
        \frac{A_k(x)\left(A_k(x)-k\cos(\pi(k+1)x)\right)}{k^3}.
\]
Thus the general local problem is reduced to evaluating the explicit function \(\Gamma_k\) on the finite Fourier grid
\[
        x=\frac1n,\frac2n,\ldots,\frac{n-1}{n}.
\]

\section{Periodic equality families}

The spectral degeneracies detected by the Fourier symbol are not merely infinitesimal in many cases.  They come from exact periodic equality families.

\begin{theorem}[Periodic equality families]
Let \(1\le k\le n-1\), and suppose that \(d\mid n\) and \(d\mid k\).  If
\[
        x_{i+d}=x_i
        \qquad
        (i=1,\ldots,n),
\]
then
\[
        S_{n,k}(x)=\frac nk.
\]
In particular, if \(1<d<n\), this gives a nonconstant family of equality points.
\end{theorem}

\begin{proof}
Let
\[
        T=x_1+\cdots+x_d.
\]
Since \(x\) has period \(d\), every block of \(k\) consecutive terms contains exactly \(k/d\) complete periods.  Hence, for every \(i\),
\[
        x_{i+1}+\cdots+x_{i+k}
        =
        \frac{k}{d}T.
\]
Therefore
\[
        \frac{x_i}{x_{i+1}+\cdots+x_{i+k}}
        =
        \frac{d}{k}\frac{x_i}{T}.
\]
Summing over one period gives
\[
        \sum_{i=1}^d
        \frac{x_i}{x_{i+1}+\cdots+x_{i+k}}
        =
        \frac{d}{k}
        \frac{x_1+\cdots+x_d}{T}
        =
        \frac dk.
\]
The full cycle consists of \(n/d\) periods.  Hence
\[
        S_{n,k}(x)
        =
        \frac nd\cdot\frac dk
        =
        \frac nk.
\]
\end{proof}

\begin{corollary}[A periodic obstruction to strict stability]
If \(\gcd(n,k)>1\), then the equal point cannot be a strict local minimum of \(S_{n,k}\) modulo scaling.  More precisely, there are nonconstant positive vectors \(x\), arbitrarily close to the equal point, such that
\[
        S_{n,k}(x)=\frac nk.
\]
Consequently, \(\gcd(n,k)=1\) is a necessary condition for strict local quadratic stability.
\end{corollary}

\begin{proof}
Let \(d=\gcd(n,k)>1\).  Then \(d\mid n\) and \(d\mid k\).  By the theorem, every positive \(d\)-periodic vector satisfies
\[
        S_{n,k}(x)=\frac nk.
\]
Taking nonconstant \(d\)-periodic vectors arbitrarily close to the constant vector gives the claim.
\end{proof}

\begin{remark}
For \(k=2\) and even \(n\), this recovers the alternating equality family \((a,b,a,b,\ldots)\).  For \(k=3\) and \(3\mid n\), it gives the exact period-three equality family \((a,b,c,a,b,c,\ldots)\).  Thus several zero Fourier modes have a global equality explanation.
\end{remark}

\section{General consequence for fixed \(k\ge3\)}

The preceding formula implies that for every fixed \(k\ge3\), the equal point is eventually a saddle as \(n\to\infty\).

\begin{proposition}
For every fixed \(k\ge3\), there exists \(N_k\) such that for all \(n\ge N_k\),
\[
        C_{n,k}^{\rm loc}<0.
\]
Hence, for every fixed \(k\ge3\) and all sufficiently large \(n\), the equal point is a saddle of \(S_{n,k}\).
\end{proposition}

\begin{proof}
Consider
\[
        \Gamma_k(x)
        =
        \frac{A_k(x)\left(A_k(x)-k\cos(\pi(k+1)x)\right)}{k^3},
        \qquad
        A_k(x)=\frac{\sin(\pi kx)}{\sin(\pi x)}.
\]
At
\[
        x_0=\frac1k
\]
one has \(A_k(x_0)=0\).  Moreover
\[
        A_k'(x_0)=\frac{\pi k\cos \pi}{\sin(\pi/k)}<0,
\]
so \(A_k(x)<0\) for \(x>x_0\) sufficiently close to \(x_0\).  Also
\[
        \cos(\pi(k+1)x_0)
        =
        \cos\left(\pi+\frac{\pi}{k}\right)
        =
        -\cos\frac{\pi}{k}<0.
\]
Therefore, for \(x>x_0\) sufficiently close to \(x_0\),
\[
        A_k(x)-k\cos(\pi(k+1)x)>0,
\]
and hence
\[
        \Gamma_k(x)<0.
\]
Thus there is an interval \((1/k,1/k+\varepsilon_k)\) on which \(\Gamma_k<0\).  For all sufficiently large \(n\), the grid \(\{m/n:1\le m\le n-1\}\) intersects this interval.  For such \(n\) there is a Fourier mode with negative coefficient, so \(C_{n,k}^{\rm loc}<0\).
\end{proof}

\section{A uniform eventual-saddle theorem}

The preceding proposition proves eventual saddle behaviour for each fixed \(k\ge3\).  We now give an explicit uniform version.  The constant is not optimized; its purpose is to give a simple closed threshold independent of numerical root finding.

\begin{lemma}[A uniform negative interval]
For every \(k\ge3\),
\[
        \Gamma_k(x)<0
        \qquad
        \text{whenever}
        \qquad
        \frac1k<x<\frac{13}{12k}.
\]
\end{lemma}

\begin{proof}
Write
\[
        x=\frac{1+s}{k},
        \qquad
        0<s<\frac1{12}.
\]
Then
\[
        A_k(x)
        =
        \frac{\sin(\pi(1+s))}{\sin(\pi(1+s)/k)}
        =
        -\frac{\sin(\pi s)}{\sin(\pi(1+s)/k)}<0.
\]
It remains to show that
\[
        A_k(x)-k\cos(\pi(k+1)x)>0.
\]
Put
\[
        \delta=\pi\left(s+\frac{1+s}{k}\right).
\]
Since
\[
        \pi(k+1)x
        =
        \pi(1+s)\left(1+\frac1k\right)
        =
        \pi+\delta,
\]
we have
\[
        \cos(\pi(k+1)x)=-\cos\delta.
\]
Hence
\[
        A_k(x)-k\cos(\pi(k+1)x)
        =
        A_k(x)+k\cos\delta.
\]
For \(k\ge3\) and \(0<s<1/12\),
\[
        0<\frac{\pi(1+s)}{k}\le \frac{13\pi}{36}<\frac\pi2.
\]
Using \(\sin y\ge 2y/\pi\) on \([0,\pi/2]\) and \(\sin y\le y\), we get
\[
        |A_k(x)|
        \le
        \frac{\pi s}{2(1+s)/k}
        =
        \frac{\pi k}{2}\frac{s}{1+s}
        <
        \frac{\pi k}{26}.
\]
Also
\[
        \delta
        =
        \pi\left(s+\frac{1+s}{k}\right)
        \le
        \pi\left(\frac1{12}+\frac{13}{36}\right)
        =
        \frac{4\pi}{9}.
\]
Thus
\[
        \cos\delta\ge \cos\frac{4\pi}{9}.
\]
Since
\[
        \cos\frac{4\pi}{9}>\frac\pi{26}
\]
(the numerical values are \(0.173648\ldots\) and \(0.120830\ldots\), respectively), it follows that
\[
        A_k(x)+k\cos\delta
        \ge
        k\left(\cos\frac{4\pi}{9}-\frac\pi{26}\right)>0.
\]
Therefore
\[
        A_k(x)<0,
        \qquad
        A_k(x)-k\cos(\pi(k+1)x)>0,
\]
and consequently
\[
        \Gamma_k(x)<0.
\]
\end{proof}

\begin{theorem}[Uniform eventual saddle bound]
Let \(k\ge3\).  If
\[
        n>12k,
\]
then
\[
        C_{n,k}^{\rm loc}<0.
\]
Hence the equal point is a saddle point of \(S_{n,k}\).  In particular, there are positive vectors arbitrarily close to the equal point for which
\[
        S_{n,k}(x)<\frac nk.
\]
\end{theorem}

\begin{proof}
By the preceding lemma, \(\Gamma_k(x)<0\) on
\[
        \left(\frac1k,\frac{13}{12k}\right).
\]
If \(n>12k\), then the interval
\[
        \left(\frac nk,\frac{13n}{12k}\right)
\]
has length
\[
        \frac{n}{12k}>1.
\]
Therefore it contains an integer \(m\).  For this integer,
\[
        \frac1k<\frac mn<\frac{13}{12k},
\]
so
\[
        \gamma_{m,k}=\Gamma_k\left(\frac mn\right)<0.
\]
Thus
\[
        C_{n,k}^{\rm loc}
        =
        \min_{1\le m\le n-1}\gamma_{m,k}<0,
\]
and the local saddle conclusion follows from the spectral stability theorem.
\end{proof}

\begin{remark}
The bound \(n>12k\) is only a simple uniform sufficient condition.  The sharper finite-grid criterion below gives substantially better thresholds for small and moderate \(k\).  The point of the uniform theorem is qualitative: for every denominator length \(k\ge3\), the equal point is eventually locally unstable once the cycle length is large enough.
\end{remark}

\section{A quantitative saddle criterion for arbitrary \(k\)}

The preceding proposition is qualitative.  We now record a computable quantitative version.  Let
\[
        \Gamma_k(x)
        =
        \frac{A_k(x)\left(A_k(x)-k\cos(\pi(k+1)x)\right)}{k^3},
        \qquad
        A_k(x)=\frac{\sin(\pi kx)}{\sin(\pi x)}.
\]
For \(k\ge3\), define \(\beta_k\) to be the right endpoint of the first negative interval of \(\Gamma_k\) after \(1/k\):
\[
        \beta_k
        =
        \sup\left\{
        b>\frac1k:
        \Gamma_k(x)<0\ \text{for all}\ x\in\left(\frac1k,b\right)
        \right\}.
\]
The proof above shows that \(\beta_k>1/k\).  Therefore the following criterion is immediate.

\begin{corollary}[A finite-grid saddle criterion]
Let \(k\ge3\).  If there is an integer \(m\) such that
\[
        \frac nk<m<\beta_k n,
\]
then
\[
        C_{n,k}^{\rm loc}<0.
\]
In particular, if
\[
        n\left(\beta_k-\frac1k\right)>1,
\]
then the equal point is a saddle.
\end{corollary}

\begin{proof}
The condition \(n/k<m<\beta_kn\) is exactly
\[
        \frac1k<\frac mn<\beta_k.
\]
By the definition of \(\beta_k\), this implies
\[
        \Gamma_k\left(\frac mn\right)<0.
\]
Hence one Fourier coefficient is negative, so
\[
        C_{n,k}^{\rm loc}<0.
\]
If \(n(\beta_k-1/k)>1\), then the interval \((n/k,\beta_k n)\) has length greater than one and therefore contains an integer.
\end{proof}

For the first few values of \(k\), the endpoint \(\beta_k\) and the corresponding sufficient threshold are as follows.
\[
\begin{array}{c|c|c}
\toprule
k & \beta_k & 1/(\beta_k-1/k)\\
\midrule
3 & 0.3661398 & 30.48\\
4 & 0.2902153 & 24.87\\
5 & 0.2407228 & 24.56\\
6 & 0.2057783 & 25.57\\
7 & 0.1797488 & 27.11\\
8 & 0.1595926 & 28.91\\
9 & 0.1435163 & 30.86\\
10 & 0.1303913 & 32.90\\
\bottomrule
\end{array}
\]
Thus, for example, for \(k=4\) the general criterion already implies saddle behaviour for all \(n\ge25\), and for \(k=5\) it implies saddle behaviour for all \(n\ge25\).

\begin{remark}
The finite-grid formula
\[
        C_{n,k}^{\rm loc}=\min_{1\le m\le n-1}\Gamma_k(m/n)
\]
remains the complete criterion for every pair \((n,k)\).  The number \(\beta_k\) only gives a simple sufficient test for saddle behaviour.  For small \(n\), one should use the exact finite Fourier minimum.
\end{remark}

\section{The case \(k=1\)}

For completeness, consider
\[
        S_{n,1}(x)=\sum_{i=1}^n\frac{x_i}{x_{i+1}}.
\]
Here
\[
        \lambda_{m,1}=e^{i\theta_m},
\]
and therefore
\[
        \gamma_{m,1}=1-\cos\theta_m.
\]
Thus
\[
        C_{n,1}^{\rm loc}
        =1-
        \cos\frac{2\pi}{n}>0.
\]
This is consistent with the global AM--GM inequality
\[
        \sum_{i=1}^n\frac{x_i}{x_{i+1}}\ge n.
\]

\section{The case \(k=2\)}

For \(k=2\),
\[
        \lambda_{m,2}=e^{i\theta_m}+e^{2i\theta_m}.
\]
A direct computation gives
\[
        |\lambda_{m,2}|^2=2+2\cos\theta_m,
\]
and
\[
        \Ree\lambda_{m,2}=\cos\theta_m+\cos2\theta_m.
\]
Hence
\[
        |\lambda_{m,2}|^2-2\Ree\lambda_{m,2}
        =4\sin^2\theta_m,
\]
so
\[
        \boxed{
        \gamma_{m,2}=\frac12\sin^2\theta_m
        =\frac12\sin^2\frac{2\pi m}{n}.
        }
\]
Therefore
\[
        C_{n,2}^{\rm loc}
        =
        \frac12\min_{1\le m\le n-1}
        \sin^2\frac{2\pi m}{n}.
\]

\begin{theorem}[Complete classification for \(k=2\)]
For \(k=2\),
\[
        C_{n,2}^{\rm loc}
        =
        \begin{cases}
        \dfrac12\sin^2\dfrac{\pi}{n}, & n\text{ odd},\\[1.2ex]
        0, & n\text{ even}.
        \end{cases}
\]
Consequently, the equal point is a strict local minimum for odd \(n\), and it is quadratically degenerate for even \(n\).
\end{theorem}

\begin{proof}
If \(n\) is odd, the closest nonzero grid angle to \(\pi\) is at distance \(\pi/n\).  Thus
\[
        \min_{1\le m\le n-1}\sin^2\frac{2\pi m}{n}
        =
        \sin^2\frac{\pi}{n}.
\]
If \(n\) is even, take \(m=n/2\).  Then \(\theta_m=\pi\) and \(\sin\theta_m=0\).
\end{proof}

For even \(n\), the degeneracy is not merely infinitesimal.  If
\[
        x_1=a,\quad x_2=b,\quad x_3=a,\quad x_4=b,\quad\ldots,
\]
then
\[
        \frac{x_i}{x_{i+1}+x_{i+2}}
        +
        \frac{x_{i+1}}{x_{i+2}+x_{i+3}}
        =
        \frac{a}{a+b}+\frac{b}{a+b}=1.
\]
Hence
\[
        S_{n,2}(x)=\frac n2
\]
on a nontrivial alternating family.  Thus no strict local stability inequality can hold for even \(n\).

For odd \(n\), the theorem yields the sharp local quadratic estimate: for every \(\varepsilon>0\), if \(x\) is sufficiently close to equality, then
\[
        S_{n,2}(x)-\frac n2
        \ge
        \left(
        \frac12\sin^2\frac{\pi}{n}-\varepsilon
        \right)
        \sum_{i=1}^n
        \left(\log x_i-\overline{\log x}\right)^2.
\]

\section{The case \(k=3\)}

For \(k=3\),
\[
        \lambda_{m,3}=e^{i\theta_m}+e^{2i\theta_m}+e^{3i\theta_m}.
\]
Put
\[
        c=\cos\theta_m.
\]
Then
\[
        |\lambda_{m,3}|^2
        =3+4\cos\theta_m+2\cos2\theta_m
        =(2c+1)^2,
\]
and
\[
        \Ree\lambda_{m,3}
        =\cos\theta_m+\cos2\theta_m+\cos3\theta_m
        =4c^3+2c^2-2c-1.
\]
Therefore
\[
        |\lambda_{m,3}|^2-3\Ree\lambda_{m,3}
        =2(1-c)(2c+1)(3c+2).
\]
Thus
\[
        \boxed{
        \gamma_{m,3}
        =
        \frac{2(1-\cos\theta_m)(2\cos\theta_m+1)(3\cos\theta_m+2)}{27}.
        }
\]
Since \(1-\cos\theta_m>0\) for \(m\ne0\), the sign is determined by
\[
        (2\cos\theta_m+1)(3\cos\theta_m+2).
\]
Hence
\[
        \gamma_{m,3}<0
        \quad\Longleftrightarrow\quad
        -\frac23<\cos\theta_m<-\frac12.
\]
Equivalently,
\[
        \gamma_{m,3}<0
        \quad\Longleftrightarrow\quad
        \frac{2\pi}{3}<\theta_m<\arccos\left(-\frac23\right),
\]
or symmetrically near \(2\pi\).

Set
\[
        \tau=\frac1{2\pi}\arccos\left(-\frac23\right).
\]
We shall use the elementary bounds
\[
        \frac4{11}<\tau<\frac{11}{30}.
\]
Indeed, since cosine is strictly decreasing on \([0,\pi]\), these inequalities are equivalent to
\[
        \cos\frac{8\pi}{11}>-\frac23,
        \qquad
        \cos\frac{11\pi}{15}< -\frac23,
\]
which may be verified directly, for instance by elementary Taylor bounds for cosine.
A negative mode exists if and only if there is an integer \(m\) such that
\[
        \frac n3<m<\tau n.
\]
Let
\[
        m_n=\left\lfloor\frac n3\right\rfloor+1.
\]
Then a negative mode exists if and only if
\[
        m_n<\tau n.
\]

We shall also use the following standard consequence of Niven's theorem \cite{Niven1956}: if \(r\in\mathbb Q\) and \(\cos(\pi r)\in\mathbb Q\), then
\[
        \cos(\pi r)\in\{0,\pm1/2,\pm1\}.
\]
Therefore the value \(\cos(2\pi m/n)=-2/3\) is impossible on the Fourier grid. Hence the only zero of \(\gamma_{m,3}\) away from the constant mode occurs at \(\cos(2\pi m/n)=-1/2\), equivalently when \(3\mid n\).

\subsection{Complete classification for \(k=3\)}

We now classify all \(n\ge4\).

First let \(n=3q\).  Then \(m_n=q+1\).  A negative mode exists if and only if
\[
        \frac{q+1}{3q}<\tau.
\]
By the bounds \(4/11<\tau<11/30\), the inequality fails at \(q=10\) and holds at \(q=11\).  Since the left-hand side is strictly decreasing, it holds exactly for \(q\ge11\).  For \(q=2,\ldots,10\), the mode \(m=q=n/3\) gives \(\cos\theta_m=-1/2\), so \(C_{n,3}^{\rm loc}=0\).

Next let \(n=3q+1\).  Then \(m_n=q+1\).  A negative mode exists if and only if
\[
        \frac{q+1}{3q+1}<\tau.
\]
Again the left-hand side is strictly decreasing in \(q\).  By the bounds \(4/11<\tau<11/30\), the inequality fails at \(q=6\), since \(7/19>11/30>\tau\), and holds at \(q=7\), since \(8/22=4/11<\tau\).  Hence it holds exactly for \(q\ge7\).  Thus the positive cases are
\[
        n=4,7,10,13,16,19.
\]

Finally let \(n=3q+2\).  Then \(m_n=q+1\).  A negative mode exists if and only if
\[
        \frac{q+1}{3q+2}<\tau.
\]
The left-hand side is strictly decreasing in \(q\).  By the bounds \(4/11<\tau<11/30\), the inequality fails at \(q=2\), since \(3/8>11/30>\tau\), and holds at \(q=3\), since \(4/11<\tau\).  Thus it holds exactly for \(q\ge3\).  The positive cases are
\[
        n=5,8.
\]

Combining the three cases, we obtain the complete classification.

\begin{theorem}[Complete classification for \(k=3\)]
For \(n\ge4\),
\[
        C_{n,3}^{\rm loc}>0
        \quad\Longleftrightarrow\quad
        n\in\{4,5,7,8,10,13,16,19\}.
\]
Moreover,
\[
        C_{n,3}^{\rm loc}=0
        \quad\Longleftrightarrow\quad
        n\in\{6,9,12,15,18,21,24,27,30\}.
\]
For all remaining \(n\ge4\),
\[
        C_{n,3}^{\rm loc}<0.
\]
Thus the equal point is a strict local minimum exactly for
\[
        n\in\{4,5,7,8,10,13,16,19\},
\]
is quadratically degenerate exactly for
\[
        n\in\{6,9,12,15,18,21,24,27,30\},
\]
and is a saddle for all other \(n\ge4\).
\end{theorem}

In particular, the equal point is a saddle for all \(n\ge31\).

\subsection{The first nontrivial stable example: \(n=5\)}

For \(n=5\), the minimum is attained at the modes \(m=2\) and \(m=3\).  Since
\[
        \cos\frac{4\pi}{5}=-\frac{1+\sqrt5}{4},
\]
we get
\[
        \boxed{
        C_{5,3}^{\rm loc}
        =
        \frac{5(3-\sqrt5)}{108}.
        }
\]
Consequently, for every \(\varepsilon>0\), if \(x\) is sufficiently close to equality, then
\[
        S_{5,3}(x)-\frac53
        \ge
        \left(
        \frac{5(3-\sqrt5)}{108}-\varepsilon
        \right)
        \sum_{i=1}^5
        \left(\log x_i-\overline{\log x}\right)^2.
\]
The coefficient is sharp in the local sense.

\subsection{A local counterexample: \(n=11\)}

For \(n=11\), the mode \(m=4\) satisfies
\[
        \frac{11}{3}<4<\tau\,11.
\]
Equivalently,
\[
        -\frac23<\cos\frac{8\pi}{11}<-\frac12.
\]
Therefore
\[
        \gamma_{4,3}<0.
\]
Numerically,
\[
        C_{11,3}^{\rm loc}\approx -0.00134468.
\]
If
\[
        u_i=\cos\frac{8\pi i}{11},
\]
then, for all sufficiently small nonzero \(t\),
\[
        S_{11,3}(e^{tu})<\frac{11}{3}.
\]
Thus the lower bound \(S_{11,3}\ge11/3\) fails already locally near the equal point.

\section{Numerical illustration}

The following table illustrates the local trichotomy for \(k=3\).  The word zero below means quadratic degeneracy, i.e. absence of a positive quadratic stability constant.

\[
\begin{array}{c|c|c}
\toprule
n & C_{n,3}^{\rm loc} & \text{local type of the equal point}\\
\midrule
4 & 0.148148 & \text{strict local minimum}\\
5 & 0.035367 & \text{strict local minimum}\\
6 & 0 & \text{quadratically degenerate}\\
7 & 0.066962 & \text{strict local minimum}\\
8 & 0.006355 & \text{strict local minimum}\\
9 & 0 & \text{quadratically degenerate}\\
10 & 0.035367 & \text{strict local minimum}\\
11 & -0.001345 & \text{saddle}\\
12 & 0 & \text{quadratically degenerate}\\
13 & 0.015806 & \text{strict local minimum}\\
14 & -0.003847 & \text{saddle}\\
15 & 0 & \text{quadratically degenerate}\\
16 & 0.006355 & \text{strict local minimum}\\
17 & -0.005046 & \text{saddle}\\
18 & 0 & \text{quadratically degenerate}\\
19 & 0.002959 & \text{strict local minimum}\\
20 & -0.004826 & \text{saddle}\\
\bottomrule
\end{array}
\]

The classification theorem explains the table completely.  The last strict local minimum occurs at \(n=19\), and the last zero-quadratic case occurs at \(n=30\).

\section{Small fixed values of \(k\)}

The exact finite-grid criterion can be evaluated directly for any prescribed \(k\). The following table gives illustrative positive and zero-quadratic cases for \(3\le k\le10\). It is not used in any proof. The entries were obtained by direct evaluation of the exact finite Fourier criterion; all admissible values \(n\ge k+1\) not appearing in the second or third column are saddle cases.
\[
\begin{array}{c|l|l|c}
\toprule
k & \{n:C_{n,k}^{\rm loc}>0\} & \{n:C_{n,k}^{\rm loc}=0\} & \text{saddle for all } n\ge\\
\midrule
3 & 4,5,7,8,10,13,16,19 & 6,9,12,15,18,21,24,27,30 & 31\\
4 & 5,9,13,17 & 6,8,10,12,16,20,24 & 25\\
5 & 6,7,8,11,12,16 & 10,15,20 & 21\\
6 & 7,13,19 & 8,9,12,14,18,24 & 25\\
7 & 8,9,11,15,16,22 & 14,21 & 23\\
8 & 9,17,25 & 10,12,16,18,24 & 26\\
9 & 10,11,19,20 & 12,18,27 & 28\\
10 & 11,21 & 12,15,20,22,30 & 31\\
\bottomrule
\end{array}
\]
This table is included only as an illustration of the exact finite Fourier criterion.  It illustrates that the local behaviour is genuinely arithmetic in the pair \((n,k)\).  The cases \(k=2\) and \(k=3\) admit the cleanest closed classifications, while larger fixed \(k\) are still completely decidable by the same spectral rule.

\section{The sharp Sadov constant}

We now turn from the local problem to Sadov's global normalized constant. In this section we follow Sadov's global range \(n\ge k\ge1\), where the endpoint case \(k=n\) is allowed. This differs from the local spectral sections above, where we assume \(1\le k\le n-1\).

Define
\[
        C=
        \inf_{n\ge k\ge1}\ \inf_{x_i>0}
        \frac{k}{n}
        \sum_{i=1}^n
        \frac{x_i}{x_{i+1}+\cdots+x_{i+k}},
\]
where the indices are cyclic. Sadov uses the normalization
\[
        A_{n,k}=\inf_{x_i>0} k\sum_{i=1}^n
        \frac{x_i}{x_{i+1}+\cdots+x_{i+k}}.
\]
Thus \(A_{n,k}/n\) is exactly the inner normalized quantity appearing in the definition of \(C\). Sadov proved
\[
        \frac{A_{n,k}}{n}\ge k(2^{1/k}-1).
\]
Consequently
\[
        C\ge \inf_{k\ge1}k(2^{1/k}-1)=\log2.
\]
We prove the reverse inequality by an explicit asymptotic construction; together these two facts determine the constant exactly.

For a residue class modulo \(N\), write \(|i|_N\) for its distance to \(0\), i.e.
\[
        |i|_N=\min_{q\in\mathbb Z}|i-qN|.
\]

\begin{lemma}[Concentration of the periodic exponential weights]
Let
\[
        y_i=\exp\left(N\cos\frac{2\pi i}{N}\right),
        \qquad
        P_N=\sum_{i=0}^{N-1}y_i,
        \qquad
        p_i=\frac{y_i}{P_N}.
\]
Let
\[
        s_N=\lfloor N^{2/3}\rfloor,
        \qquad
        h_N=\lfloor N^{7/12}\rfloor.
\]
Then
\[
        \max_i p_i\to0,
\]
\[
        \sum_{|i|_N>h_N}p_i\to0,
\]
and
\[
        \sum_{s_N-h_N\le |i|_N\le s_N+h_N}p_i\to0.
\]
\end{lemma}

\begin{proof}
For \(|t|\le\pi\) there is an absolute constant \(c>0\) such that
\[
        1-\cos t\ge c t^2.
\]
Thus, for \(|i|_N\le N/2\),
\[
        y_i
        \le
        e^N\exp\left(-c\frac{i^2}{N}\right),
\]
with \(i\) represented symmetrically.  On the other hand, for \(|i|\le c_0\sqrt N\), with fixed sufficiently small \(c_0>0\), one has
\[
        y_i\ge c_1 e^N.
\]
Therefore
\[
        P_N\ge c_2 e^N\sqrt N.
\]
Consequently
\[
        p_i\le \frac{C}{\sqrt N}
        \exp\left(-c\frac{|i|_N^2}{N}\right).
\]
This immediately gives \(\max_i p_i=O(N^{-1/2})\to0\).  Since
\[
        \frac{h_N}{\sqrt N}=N^{1/12}\to\infty,
\]
the Gaussian tail bound gives
\[
        \sum_{|i|_N>h_N}p_i\to0.
\]
Finally,
\[
        s_N-h_N\sim N^{2/3},
        \qquad
        \frac{s_N-h_N}{\sqrt N}\to\infty,
        \qquad
        s_N+h_N=o(N),
\]
so the same tail estimate gives
\[
        \sum_{s_N-h_N\le |i|_N\le s_N+h_N}p_i\to0.
\]
\end{proof}

\begin{lemma}[A discrete logarithmic sum]
Let \(a_N\le b_N\) be integers and let \(Z_i\ge1\) be positive numbers satisfying
\[
        Z_{i-1}\ge Z_i
        \qquad (a_N\le i\le b_N),
\]
\[
        \max_{a_N\le i\le b_N}(Z_{i-1}-Z_i)\to0,
\]
and
\[
        \sum_{i=a_N}^{b_N}(Z_{i-1}-Z_i)=O(1).
\]
Then
\[
        \sum_{i=a_N}^{b_N}\frac{Z_{i-1}-Z_i}{Z_i}
        =
        \log Z_{a_N-1}-\log Z_{b_N}+o(1).
\]
\end{lemma}

\begin{proof}
Put
\[
        r_i=\frac{Z_{i-1}-Z_i}{Z_i}\ge0.
\]
Since \(Z_i\ge1\), we have \(\max_i r_i\to0\) and
\[
        \sum_i r_i=O(1).
\]
Taylor's formula gives, uniformly for \(r_i\to0\),
\[
        \log(1+r_i)=r_i+O(r_i^2).
\]
Moreover,
\[
        \sum_i r_i^2
        \le
        \left(\max_i r_i\right)\sum_i r_i
        \to0.
\]
Hence
\[
        \sum_{i=a_N}^{b_N}r_i
        =
        \sum_{i=a_N}^{b_N}\log(1+r_i)+o(1).
\]
But
\[
        \log(1+r_i)
        =
        \log\frac{Z_{i-1}}{Z_i}
        =
        \log Z_{i-1}-\log Z_i.
\]
The sum telescopes, proving the claim.
\end{proof}

\begin{theorem}[Sharp value of the global constant]
With \(C\) as above,
\[
        C=\log2.
\]
\end{theorem}

\begin{proof}
It is enough to prove \(C\le\log2\).  Let \(N\to\infty\), put
\[
        n=2N,
        \qquad
        k=N+s_N,
        \qquad
        s_N=\lfloor N^{2/3}\rfloor.
\]
Define
\[
        x_i=\exp\left(N\cos\frac{2\pi i}{N}\right),
        \qquad i=0,1,\ldots,2N-1.
\]
Then \(x_{i+N}=x_i\).  On one period write
\[
        y_i=\exp\left(N\cos\frac{2\pi i}{N}\right),
        \qquad i\in\mathbb Z/N\mathbb Z,
\]
\[
        P_N=\sum_{i=0}^{N-1}y_i,
        \qquad
        p_i=\frac{y_i}{P_N}.
\]
Thus \(p_i\) is a probability distribution on \(\mathbb Z/N\mathbb Z\).  Define
\[
        W_i=\sum_{r=1}^{s_N}p_{i+r},
\]
with indices modulo \(N\).  Since the sequence is \(N\)-periodic, the block
\[
        i+1,i+2,\ldots,i+N
\]
contains exactly one full period.  The remaining terms
\[
        i+N+1,\ldots,i+N+s_N
\]
coincide, by periodicity, with the window
\[
        i+1,\ldots,i+s_N
\]
on one period.  Hence each denominator of length \(N+s_N\) contains one full period and one additional window of length \(s_N\):
\[
        x_{i+1}+\cdots+x_{i+N+s_N}
        =P_N(1+W_i).
\]
Therefore
\[
        S_{2N,N+s_N}(x)
        =2\sum_{i=0}^{N-1}\frac{p_i}{1+W_i}.
\]
Consequently
\[
        \frac{N+s_N}{2N}S_{2N,N+s_N}(x)
        =
        \left(1+\frac{s_N}{N}\right)T_N,
\]
where
\[
        T_N=\sum_{i=0}^{N-1}\frac{p_i}{1+W_i}.
\]
Since \(s_N/N\to0\), it remains to prove
\[
        T_N\to\log2.
\]

Let
\[
        h_N=\lfloor N^{7/12}\rfloor,
        \qquad
        a_N=-h_N,
        \qquad
        b_N=h_N.
\]
We identify the interval \([a_N,b_N]\) with its natural representatives in \(\mathbb Z/N\mathbb Z\).  Since
\[
        h_N=o(s_N),
        \qquad
        s_N+h_N=o(N),
\]
there is no wrap-around in the windows used below, for all sufficiently large \(N\).

By the concentration lemma,
\[
        \sum_{i\notin[a_N,b_N]}p_i\to0.
\]
Since \(1+W_i\ge1\), the contribution of the complement of \([a_N,b_N]\) to \(T_N\) is \(o(1)\).  Thus
\[
        T_N=
        \sum_{i=a_N}^{b_N}\frac{p_i}{1+W_i}+o(1).
\]

The window defining \(W_{a_N-1}\) is
\[
        (a_N-1,\,a_N-1+s_N]
        =[-h_N,\,-h_N-1+s_N],
\]
and this interval contains \([-h_N,h_N]\) for all large \(N\).  Hence, by the concentration lemma,
\[
        W_{a_N-1}\to1.
\]
Similarly, the window defining \(W_{b_N}\) is
\[
        (h_N,\,h_N+s_N],
\]
which is disjoint from the central concentration region and lies in the tail; hence
\[
        W_{b_N}\to0.
\]

For \(a_N\le i\le b_N\) we have the exact identity
\[
        W_{i-1}-W_i=p_i-p_{i+s_N}.
\]
Moreover,
\[
        s_N-h_N>|i|,
        \qquad
        s_N+h_N=o(N),
\]
so \(|i+s_N|_N>|i|_N\) and both residues lie in the range where \(p_j\) decreases with distance from \(0\).  Therefore
\[
        p_i\ge p_{i+s_N},
        \qquad a_N\le i\le b_N,
\]
for all large \(N\).  Also, by the concentration lemma,
\[
        \sum_{i=a_N}^{b_N}p_{i+s_N}
        \le
        \sum_{s_N-h_N\le |j|_N\le s_N+h_N}p_j
        \to0.
\]
Thus
\[
        T_N
        =
        \sum_{i=a_N}^{b_N}\frac{W_{i-1}-W_i}{1+W_i}+o(1).
\]

Put
\[
        Z_i=1+W_i.
\]
Then \(Z_i\ge1\), and on \([a_N,b_N]\)
\[
        Z_{i-1}-Z_i=W_{i-1}-W_i\ge0.
\]
Moreover
\[
        Z_{i-1}-Z_i=p_i-p_{i+s_N}\le p_i,
\]
so
\[
        \max_{a_N\le i\le b_N}(Z_{i-1}-Z_i)\le \max_i p_i\to0.
\]
Also
\[
        \sum_{i=a_N}^{b_N}(Z_{i-1}-Z_i)
        =W_{a_N-1}-W_{b_N}=O(1).
\]
The discrete logarithmic lemma gives
\[
        \sum_{i=a_N}^{b_N}\frac{W_{i-1}-W_i}{1+W_i}
        =
        \log Z_{a_N-1}-\log Z_{b_N}+o(1).
\]
Since
\[
        Z_{a_N-1}=1+W_{a_N-1}\to2,
        \qquad
        Z_{b_N}=1+W_{b_N}\to1,
\]
we get
\[
        T_N\to\log2.
\]
Consequently
\[
        \frac{N+s_N}{2N}S_{2N,N+s_N}(x)\to\log2.
\]
Thus \(C\le\log2\).  Combining this with Sadov's lower bound \(C\ge\log2\), we obtain \(C=\log2\).
\end{proof}
\begin{remark}
The proof shows where the logarithm appears.  The additional window of length \(s_N\) moves across a sharply concentrated period.  Its normalized mass drops from \(1\) to \(0\), and the main sum becomes a discrete logarithmic integral
\[
        \sum \frac{dW}{1+W}\to \int_0^1\frac{dW}{1+W}=\log2.
\]
\end{remark}

\section{Relation between the global and local results}

The identity \(C=\log2\) concerns the global normalized infimum over all admissible triples \((n,k,x)\).  It does not settle the individual fixed-\(k\) asymptotic constants, which are separate and generally much harder problems.  The Fourier analysis in the preceding sections concerns a different question: the local behaviour of each fixed sum \(S_{n,k}\) near its equal point.  These two parts are complementary.

The global construction in the proof of \(C=\log2\) is not a small perturbation of the equal point.  It uses highly concentrated periodic weights.  By contrast, the local theory describes the Hessian at equality and identifies when equality is locally stable, quadratically degenerate, or unstable.  In particular, when \(C_{n,k}^{\rm loc}<0\), the equal point is not even locally minimizing, so global extremal configurations cannot be controlled by a neighbourhood of equality.  When \(C_{n,k}^{\rm loc}>0\), any global descent below \(n/k\) must occur away from equality.

\section{Conclusion}

We proved two complementary results for Shapiro--Diananda cyclic sums.  First, the global Sadov constant satisfies
\[
        C=\log2.
\]
The upper bound is obtained by an explicit asymptotic construction with two identical sharply concentrated periods and window length \(N+s_N\), where \(s_N\to\infty\), \(s_N/\sqrt N\to\infty\), and \(s_N/N\to0\).  The additional window moves across one concentration peak and produces the limiting logarithmic integral \(\int_0^1(1+u)^{-1}\,du=\log2\).

Second, we computed the exact local quadratic stability of
\[
        S_{n,k}(x)=\sum_{i=1}^n\frac{x_i}{x_{i+1}+\cdots+x_{i+k}}
\]
at the equal point.  The logarithmic parametrization \(x_i=e^{u_i}\) turns the Hessian into a circulant quadratic form.  Its Fourier symbol is
\[
        \gamma_{m,k}
        =
        \frac{|\lambda_{m,k}|^2-k\Ree\lambda_{m,k}}{k^3},
        \qquad
        \lambda_{m,k}=\sum_{j=1}^k e^{2\pi imj/n}.
\]
This gives the sharp local quadratic coefficient
\[
        C_{n,k}^{\rm loc}=\min_{1\le m\le n-1}\gamma_{m,k}.
\]
The sign of this constant gives a complete local trichotomy: strict local minimum, quadratic degeneracy, or saddle.

The cases \(k=2\) and \(k=3\) were classified completely.  For \(k=2\), the equal point is locally stable exactly when \(n\) is odd.  For \(k=3\), it is locally stable only for
\[
        n\in\{4,5,7,8,10,13,16,19\},
\]
quadratically degenerate for
\[
        n\in\{6,9,12,15,18,21,24,27,30\},
\]
and a saddle for all other \(n\ge4\).  For every fixed \(k\ge3\), the equal point is a saddle for all sufficiently large \(n\); we also give an explicit finite-grid saddle criterion and small-\(k\) classifications.


\begin{thebibliography}{9}

\bibitem{Diananda1962}
P. H. Diananda,
\newblock Some cyclic and other inequalities,
\newblock \emph{Mathematical Proceedings of the Cambridge Philosophical Society} \textbf{58} (1962), no. 2, 184--190.

\bibitem{Drinfeld1971}
V. G. Drinfel'd,
\newblock A cyclic inequality,
\newblock \emph{Mathematical Notes of the Academy of Sciences of the USSR} \textbf{9} (1971), 68--71.

\bibitem{Troesch1989}
B. A. Troesch,
\newblock The validity of Shapiro's cyclic inequality,
\newblock \emph{Mathematics of Computation} \textbf{53} (1989), no. 188, 657--664.

\bibitem{BushellMcLeod2002}
P. J. Bushell and J. B. McLeod,
\newblock Shapiro's cyclic inequality for even \(n\),
\newblock \emph{Journal of Inequalities and Applications} \textbf{7} (2002), no. 3, 331--348.

\bibitem{Niven1956}
I. Niven,
\newblock \emph{Irrational Numbers},
\newblock Carus Mathematical Monographs, no. 11, Mathematical Association of America, 1956.

\bibitem{Sadov2016}
S. Sadov,
\newblock Lower bound for cyclic sums of Diananda type,
\newblock \emph{Archiv der Mathematik} \textbf{106} (2016), no. 2, 135--144.

\bibitem{Sadov2021}
S. Sadov,
\newblock Three steps away from Shapiro's problem: lower bounds for graphic sums with functions `max' or `min' in denominators,
\newblock arXiv:2106.10877, 2021.

\bibitem{Sadov2022}
S. Sadov,
\newblock Beyond Shapiro's problem: from cyclic sums to ``graphic'' sums,
\newblock arXiv:2212.05968, 2022.

\end{thebibliography}
\end{document}